\documentclass[12pt,reqno]{amsart}
\usepackage{amsmath}
\usepackage{amssymb}
\usepackage{amstext}
\usepackage{comment}
\usepackage{bm}
\usepackage{a4wide}
\usepackage{graphicx}
\usepackage[numbers,sort&compress]{natbib}
\allowdisplaybreaks \numberwithin{equation}{section}
\usepackage{color}
\usepackage{cases}

\numberwithin{equation}{section}

\newtheorem{theorem}{Theorem}[section]
\newtheorem{proposition}[theorem]{Proposition}

\newtheorem{lemma}[theorem]{Lemma}
\newtheorem*{Arnold's Second Stability Theorem}{Arnold's Second Stability Theorem}
\newtheorem*{Burton's Stability Criterion}{Burton's Stability Criterion}

\theoremstyle{definition}

\newtheorem{definition}[theorem]{Definition}

\theoremstyle{remark}
\newtheorem{remark}[theorem]{Remark}

\newtheorem{example}[theorem]{Example}

\begin{document}

\title[On the positive constant in Arnold's second stability theorem]{On the positive constant in Arnold's second stability theorem for a bounded  domain}

\author{Fatao Wang}
\address{School of Mathematical Sciences, Dalian University of Technology, Dalian 116024, PR China}
\email{wangfatao@mail.dlut.edu.cn}

\author{Guodong Wang}
\address{School of Mathematical Sciences, Dalian University of Technology, Dalian 116024, PR China}
\email{gdw@dlut.edu.cn}

\author{Bijun Zuo}
\address{College of Mathematical Sciences, Harbin Engineering University, Harbin {\rm150001}, PR China}
\email{bjzuo@amss.ac.cn}


\begin{abstract}
For a steady flow of a two-dimensional ideal fluid, the gradient vectors of the stream function $\psi$ and its vorticity $\omega$ are collinear. Arnold's second stability theorem states that the flow is Lyapunov stable if $0<\nabla\omega/\nabla\psi<C_{ar}$ for some $C_{ar}>0$. In this paper, we show that, for a bounded domain, $C_{ar}$ can be taken as the first eigenvalue $\bm\Lambda_1$ of a certain Laplacian eigenvalue problem. When  $\nabla\omega/\nabla\psi$ reaches $\bm\Lambda_1$, instability may occur, as illustrated by a non-circular steady flow in a disk; however, a certain form of structural stability still holds. Based on these results, we establish a theorem on the rigidity and orbital stability of steady Euler flows in a disk.
\end{abstract}

\maketitle

\tableofcontents

\section{Introduction and main results}\label{Sec1}

Stability theory in fluid dynamics examines how a fluid flow reacts to small perturbations, determining whether it remains close to its original state (stable) or diverges (unstable). For ideal (incompressible and inviscid) fluid flows governed by the Euler equations, stability analysis plays a crucial role in understanding the persistence of flow patterns such as vortices, jets, and other coherent structures.
In the setting of two-dimensional (2D) ideal flows, among the most important and useful stability results are Arnold's stability theorems, including  the first stability theorem and the second stability theorem, both established by Arnold \cite{A1,A2} in the 1960s. Roughly speaking, Arnold's stability theorems assert that  a steady flow of a two-dimensional ideal fluid is stable if its stream function $\psi$ and its vorticity $\omega$ satisfies one of the following two conditions:
 \begin{itemize}
\item[(A1)] $\nabla\omega/\nabla\psi<0$ (the first stability theorem), or
\item [(A2)]
$0<\nabla\omega/\nabla\psi<C_{ar}$
 for some sufficiently small positive constant $C_{ar}$ (the second stability theorem).
 \end{itemize}
 The first stability theorem is relatively simple, and its most general form can be obtained by repeating the argument in \cite[Section 5]{CWZ}. In contrast, the second stability theorem is significantly more complicated, particularly regarding its possible extensions \cite{WG96,WG98,W22,WZ23} and limitations \cite{And,WS00}.
This paper focuses on the positive constant
$C_{ar}$ for a smooth, bounded domain, whether it is simply or multiply connected, aiming to clarify the scope of application of the second stability theorem.

 \subsection{2D Euler equations in a bounded domain}
Throughout this paper, let $D\subset\mathbb R^2$ be a smooth, bounded domain.   Suppose that $\partial D$ has $N+1$ connected components, denoted by $\Gamma_0,\Gamma_1,\dots,\Gamma_N$,
where $\Gamma_0$ is the outer boundary, and $\Gamma_1,\dots,\Gamma_N$ are the $N$ inner boundaries.
Note that we allow $N=0,$ which means that $D$ is simply connected.

The motion of an ideal fluid in $D$ with impermeability boundary condition is described by the following Euler equations:
\begin{equation}\label{euler}
\begin{cases}
\partial_t\mathbf v+(\mathbf  v\cdot \nabla)\mathbf v=-\nabla P,&t\in\mathbb R,\ \mathbf x=(x_1,x_2)\in D,\\
\nabla\cdot\mathbf v=0,&t\in\mathbb R,\ \mathbf x\in D,\\
\mathbf v\cdot\mathbf n=0,&t\in\mathbb R,\ \mathbf x\in \partial D,\\
\mathbf v|_{t=0}=\mathbf v_0,&\mathbf x\in D,\\
\end{cases}
\end{equation}
where $\mathbf v(t,\mathbf x)=(v_1(t,\mathbf x), v_2(t,\mathbf x))$ is the velocity field, $P(t,\mathbf x)$ is the scalar pressure, and $\mathbf n(\mathbf x)$ is the outward (with respect to $D$) unit normal at $\mathbf x\in\partial D.$

We now derive the stream function formulation and the vorticity–circulation pair formulation of the Euler equations \eqref{euler}.
By the divergence-free condition $\nabla\cdot\mathbf v=0$ and the impermeability boundary condition $\mathbf v\cdot\mathbf n|_{\partial D}=0$, there exists a scalar function $\psi$ that is piecewise constant on the boundary, called the \emph{stream function}, such that
\[
\mathbf v=\nabla^\perp\psi,\quad \nabla^\perp:=(\partial_{x_2}, -\partial_{x_1}).
\]
Note that the stream function can differ by an arbitrary constant.
\emph{Throughout this paper, we always assume that the stream function has zero mean.}
Let $\omega:=\partial_{x_1}v_2-\partial_{x_2}v_1$ be the scalar  \emph{vorticity} of the fluid, and   $\bm\gamma=(\gamma_1,\dots,\gamma_N)\in\mathbb R^{N}$  be the \emph{circulation vector} around the $N$ inner boundary components:
\[\gamma_i=\int_{\Gamma_i}\mathbf v\cdot\mathbf r dS=-\int_{\Gamma_i}\frac{\partial \psi}{\partial\mathbf n}dS, \quad i=1,\dots,N,\]
where $\mathbf{r}$ denotes the anticlockwise rotation of $\mathbf{n}$ by $\pi/2$, and $dS$ denotes the length element on $\partial D$.
Note that $\bm\gamma$ is conserved by Kelvin's circulation theorem.
One can check that $\psi$ satisfies
\begin{equation}\label{psieq0}
\begin{cases}
-\Delta \psi=\omega,&\mathbf x\in D,\\
\psi|_{\Gamma_{i}}={\rm constant},&i= 0,1,\dots,N,\\
\int_{\Gamma_{i}}\nabla\psi\cdot \mathbf n dS=-\gamma_i,&i= 1,\dots,N,\\
\int_D\psi d\mathbf x=0.
\end{cases}
 \end{equation}
Note that the constant value of $\psi$ on each boundary component may vary with time.
By Proposition \ref{prop21} in Section \ref{Sec2}, given $\omega$ and $\bm\gamma=(\gamma_1,\dots,\gamma_N)\in\mathbb R^{N}$, the elliptic problem \eqref{psieq0} has a unique solution. Therefore,  \emph{the motion of an ideal fluid in $D$ can be described from three independent perspectives: the velocity field $\mathbf v$,  the stream function $\psi$, and the vorticity-circulation pair $(\omega,\bm\gamma).$} Note that  the circulation vector $\bm\gamma$ disappears when $D$ is simply connected. In terms of the stream function $\psi$, the Euler equations \eqref{euler} can be written as
\begin{equation}\label{eqpsi}
\partial_t(\Delta \psi)+\nabla^\perp\psi\cdot\nabla(\Delta\psi)=0.
\end{equation}

Conservation laws of the Euler equations play an essential role in this paper. In addition to the circulation vector $\bm\gamma$ mentioned above, the \emph{kinetic energy} $E$ is also a conserved quantity:
\begin{equation}\label{kien}
E(t)=E(0)\quad\forall\, t\in\mathbb R,\quad E(t):=\frac{1}{2}\int_D|\mathbf v(t,\mathbf x)|^2 d\mathbf x
\end{equation}
Another important conservation law is the  \emph{isovortical property}, which states that the distribution function of the vorticity remains unchanged with time (see \cite[Chapter 1]{MB}).
To describe the isovortical property more concisely, we introduce some notation.
Given two Lebesgue measurable functions $v,w:D\mapsto \mathbb R,$ we say that $v$ is  \emph{equimeasurable} with $w$, written as $v\sim w$, if they satisfy
\begin{equation}\label{simdf}
\left|\left\{\mathbf x\in D \;\middle|\;  v(\mathbf x)>s\}|=|\{\mathbf x\in D \;\middle|\;  w(\mathbf x)>s\right\}\right|\quad\forall\, s\in\mathbb R.
\end{equation}
  Here and hereafter, $|A|$ denotes the two-dimensional Lebesgue measure of any Lebesgue measurable set $A$.
  For a measurable function $v: D\mapsto \mathbb R$, denote by $\mathcal R_v$ the \emph{rearrangement class} of $v$; that is, the set of all measurable functions that are equimeasurable with $v,$
\begin{equation}\label{rcdf}
\mathcal R_v:=\left\{w: D\mapsto \mathbb R \mbox{ is Lebesgue measurable} \;\middle|\;  w\sim v\right\}.
\end{equation}
Then the isovortical property can be written as
\begin{equation}\label{isovp}
\omega(t,\cdot)\in\mathcal R_{\omega_0}\quad\forall\, t\in\mathbb R.
\end{equation}

 \subsection{Arnold's second stability theorem}
 As can be easily seen from \eqref{eqpsi}, a flow is steady if and only if  the gradient vectors of its stream function and its vorticity are collinear.  For this to hold, a sufficient condition is that the stream function $\psi^s$ satisfies
\begin{equation}\label{semilieq}
\begin{cases}
-\Delta  \psi^s=g(\psi^s),&\mathbf x\in D,\\
\psi^s|_{\Gamma_i}={\rm constant},& i=0,1,\dots,N,\\
\int_D\psi^s d\mathbf x=0
\end{cases}
\end{equation}
for some $g\in C^1(\mathbb R)$.

To study the Lyapunov stability of  the steady flow  \eqref{semilieq}, Arnold \cite{A1,A2}
introduced  the following energy-Casimir ($EC$)  functional:
 \[EC(\psi)=\frac{1}{2}\int_D|\nabla \psi|^2d\mathbf x-\int_DF(-\Delta \psi)d\mathbf x-\sum_{i=0}^N\psi^s\big|_{\Gamma_i}\int_{\Gamma_i}
 \frac{\partial(\psi-\psi^s)}{\partial\mathbf n}dS,\ \ F(s) :=\int_0^s g^{-1}(\tau)d\tau.\]
Here, $g$ is assumed to be invertible.
 The steady condition \eqref{semilieq} ensures $\delta EC|_{\psi^s}=0$. As in the finite-dimensional case, stability is expected if $\delta^2 EC|_{\psi^s}$ is positive or negative definite. By computing $\delta^2 EC|_{\psi^s}$, Arnold established the following theorem, the detailed proof of which can also be found in \cite{AK, MP}.

 \begin{theorem}[Arnold's second stability theorem, \cite{A1,A2}]\label{asst0}
Suppose that
$\psi^s\in C^{2}(\bar D)$ solves \eqref{semilieq}, where $g\in C^{1}(\mathbb R)$ satisfies
\begin{equation}\label{a2tj1}
 \min_{\bar D}g'(\psi^s)>0.
\end{equation}
Then there exists some $C_{ar}>0$, such that if
\begin{equation}\label{a2tj2}
  \max_{\bar D}g'(\psi^s)<C_{ar},
\end{equation}
then the associated steady flow  is  stable in the following sense: for any $\varepsilon>0,$ there exists some $\delta>0,$ such that for any smooth Euler flow with stream function $\psi,$ it holds that
 \[\|\psi(0,\cdot)-\psi^s\|_{H^2(D)}<\delta\quad \Longrightarrow\quad
 \|\psi(t,\cdot)-\psi^s\|_{H^2(D)}<\varepsilon\quad\forall\, t\in\mathbb R.\]
 \end{theorem}

 \begin{remark}
In the original form of Arnold's stability theorems (both the first and the second), conditions are imposed on the ratio  $\nabla\psi^s/\nabla\Delta\psi^s$. However, it was tacitly assumed in the proofs that the vorticity should be a monotone function of the stream function.
 \end{remark}

\begin{remark}
By examining the proof of Theorem \ref{asst0},  one can take $C_{ar}=\mathsf m,$ where
\begin{equation}\label{infimum}
\mathsf m:=\inf\left\{\frac{\int_D|\Delta u|^2d\mathbf x}{\int_D|\nabla u|^2d\mathbf x}\;\middle|\; u\in H^2(D), \ u|_{\Gamma_i}=c_i,\ i=0, 1,\dots,N\right\}.
\end{equation}
\emph{Here and hereafter, $c_i$ denotes unprescribed constants.}
One can readily verify that $\mathsf m=\lambda_1$
when $D$ is simply connected,  where $\lambda_k$ denotes the $k$-th eigenvalue of the zero-Dirichlet Laplacian eigenvalue problem:
\begin{equation}\label{zerolpls}
\begin{cases}
  -\Delta u=\lambda u, & \mathbf x\in D, \\
  u=0, & \mathbf x\in \partial D.
\end{cases}
\end{equation}
\end{remark}

Arnold's theory has been developed by many authors. Some closely related works are listed below in chronological order:
\begin{itemize}
\item In the 1990s, Wolansky and Ghil \cite{WG96,WG98} developed the  supporting functional method, and showed that \eqref{a2tj2} can be replaced by the following weaker condition when $D$ is simply connected:
\begin{equation}\label{c2b}
   \min_{u\in H^{1}_{0}(D),\ \|u\|=1} \int_{D}|\nabla u|^{2}-g'(\psi^s)u^2 d\mathbf x+\|\mathrm Pu\|^2>0,
\end{equation}
where $\|\cdot\|$ is the norm induced by the inner product $<\cdot,\cdot>_{g'(\psi^s)}$,
\[<u_1,u_2>_{g'(\psi^s)}:=\int_Du_1u_2g'(\psi^s)d\mathbf x,\]
and $\mathrm P$ is the projection operator with respect to  $<\cdot,\cdot>_{g'(\psi^s)}$  onto the associated Casimir space,  the $L^2$-closure of
\[\left\{u:D\to \mathbb R \;\middle|\; u=h( \psi),\ h\in C^1(\mathbb R)\right\}.\]
See also \cite{Lin04} for an equivalent formulation of the condition \eqref{c2b} in terms of level sets of the stream function.
\item Inspired by the works of Wolansky and Ghil \cite{WG96,WG98} and Burton's rearrangement theory \cite{B87,B05}, Wang \cite{W22} showed that, when $D$ is simply connected,  stability still holds  if \eqref{a2tj1} is relaxed to $\min_{\bar D}g'(\psi^s)\geq 0$ and  \eqref{a2tj2} is relaxed to
 \begin{equation}\label{c2w}
   \min_{u\in H^{1}_{0}(D),\ \|u\|_{L^2(D)}=1} \int_{D}|\nabla u|^{2}-g'(\psi^s)u^2 d\mathbf x\geq 0.
\end{equation}
Moreover, the stability in \cite{W22}  is measured in terms of the $L^p$-norm of the vorticity (or, equivalently, the $W^{2,p}$-norm of the stream function) for any $1<p<\infty$, which is more general.
\item
Subsequently, Wang and Zuo \cite{WZ22} extended the
 result  of \cite{W22} to multiply connected domains.
A straightforward corollary of \cite[Theorem 1.5]{WZ22} is that the assumption \eqref{a2tj2} in Theorem \ref{asst0} can be replaced by
\begin{equation}\label{leq96}
  \max_{\bar D}g'(\psi^s)\leq \Lambda_1,
\end{equation}
where
\begin{equation}\label{caplam1}
 \Lambda_1 := \inf_{u \in \mathcal{Y},\, u \not\equiv 0}
\frac{\int_D |\Delta u|^2  d\mathbf{x}}{\int_D |\nabla u|^2 d\mathbf{x}},
\end{equation}
and
\[
\mathcal{Y} := \left\{ u \in H^2(D) \;\middle|\; u|_{\Gamma_0} = 0,\
u|_{\Gamma_i} = c_i,\
\int_{\Gamma_i} \frac{\partial u}{\partial \mathbf{n}}  dS = 0,\ i=1,\dots,N \right\}.
\]
It is clear that $\Lambda_1\geq \mathsf m,$ where $\mathsf m$ is given by \eqref{infimum}.
\item Recently, Wang \cite{W24j}
 showed that when $D$ is a disk,
the steady flow related to any first eigenfunction of
\begin{equation}\label{egendk}
\begin{cases}
 -\Delta u=\bm\Lambda u, & \mathbf x\in D, \\
  u|_{\partial D}={\rm constant},\\
  \int_D ud\mathbf x=0
\end{cases}
\end{equation}
is \emph{orbitally stable} up to rigid rotation. Note that instability may occur for these first eigenstates (see Example \ref{examins} below), which means that
$C_{ar}$ cannot exceed the first eigenvalue of \eqref{egendk}.
\end{itemize}

  \subsection{Main results}

To begin with, we study the following Laplacian eigenvalue problem, which serves as the foundation for the subsequent discussion:
 \begin{equation}\label{lepzm}
 \begin{cases}
 -\Delta u=\bm\Lambda u,&\mathbf x\in D,\\
 u\in\mathbf Y,
 \end{cases}
 \end{equation}
 where
  \begin{equation}\label{dfyy}
\mathbf Y:=\left\{u\in H^2(D) \;\middle|\; \int_D ud\mathbf x=0,\ u|_{\Gamma_i}=c_i,\
 \int_{\Gamma_i}\frac{\partial u}{\partial \mathbf n} dS=0, \ i=0,1,\dots,N
  \right\}.
  \end{equation}
Note that  \eqref{lepzm} is mainly inspired by \cite{W24j}; when $D$ is simply connected, \eqref{lepzm} reduces to \eqref{egendk}.

 \begin{theorem}[Laplacian eigenvalue problem]\label{thm0}
 The following statements hold for  \eqref{lepzm}:
 \begin{itemize}
 \item[(i)] The eigenvalues, counted without multiplicity, are all real and positive. They can be arranged as
\[0<\bm\Lambda_1<\bm\Lambda_2<\dots\to+\infty,\]
and the eigenspace $\mathbf E_k$  associated with each $\bm\Lambda_k$ is finite-dimensional.
\item[(ii)] For any $u\in\mathbf Y$, it holds that
\begin{equation}\label{infory}
\bm\Lambda_1\int_Du^2d\mathbf x\leq \int_D|\nabla u|^2d\mathbf x,
\end{equation}
and the equality holds if and only if $u\in\mathbf E_1$.
\item[(iii)]
  $\bm\Lambda_1>\Lambda_1$, where $\Lambda_1$ is given by \eqref{caplam1}.
  \end{itemize}
 \end{theorem}

Throughout this paper, let $1<p<\infty$ be fixed. The following theorem shows that the constant $C_{ar}$ in Theorem \ref{asst0} can be taken as $\bm\Lambda_1$, and that the norm measuring stability may be chosen more generally.
\begin{theorem}[Stability: $g'<\bm\Lambda_1$]\label{thm1}
Suppose that
$\psi^s\in C^{2}(\bar D)$ solves \eqref{semilieq}, where $g\in C^{1}(\mathbb R)$. Suppose that
\begin{equation}\label{condg}
 0\leq \min_{\bar D} g' (\psi^s)\leq \max_{\bar D}g' (\psi^s)<\bm\Lambda_1.
 \end{equation}
Denote by $(\omega^s,\bm\gamma^s)$ the vorticity-circulation pair related to $\psi^s$, i.e.,
\begin{equation}\label{vcp01}
\omega^s=-\Delta\psi^s,\quad \bm\gamma^s=(\gamma_1,\dots,\gamma_N),\quad
\gamma_i:=-\int_{\Gamma_i}\frac{\partial \psi^s}{\partial \mathbf n}dS.
\end{equation}
Then, for any $\varepsilon>0$, there exists some $\delta>0$ such that for any  smooth Euler flow with vorticity-circulation pair $(\omega,\bm\gamma)$, the following holds:
 \[\|\omega(0,\cdot)-\omega^s\|_{L^p(D)}+|\bm\gamma-\bm\gamma^s|<\delta
 \quad\Longrightarrow\quad\|\omega(t,\cdot)-\omega^s\|_{L^p(D)}<\varepsilon\quad \forall \, t\in\mathbb R,\]
 where $|\bm\gamma-\bm\gamma^s|$ denotes the Euclidean distance between  $\bm\gamma$ and $\bm\gamma^s$ in $\mathbb R^N.$
\end{theorem}

\begin{remark}
By \eqref{leq96} and Theorem \ref{thm0}(iii), Theorem \ref{thm1} is an extension of Theorem \ref{asst0}.
\end{remark}
\begin{remark}\label{vpfor1}
According to Proposition \ref{prop21} in Section \ref{Sec2}, the conclusion of Theorem \ref{thm1} can also be stated in terms of the stream function as follows:  {
for any $\varepsilon>0$, there exists some $\delta>0,$ such that for any smooth Euler flow with stream function $\psi$, it holds that
 \[\|\psi(0,\cdot)-\psi^s\|_{W^{2,p}(D)}<\delta\quad\Longrightarrow\quad \|\psi(t,\cdot)-\psi^s\|_{W^{2,p}(D)}<\varepsilon\quad \forall\, t\in\mathbb R.\]
}
\end{remark}

\begin{remark}
The smoothness of the perturbed flow in Theorem \ref{thm1}, as well as in Theorem \ref{thm2} below, can be relaxed. Specifically, it suffices to require that the vorticity-circulation pair of the perturbed flow is an \emph{admissible pair}; see  Definition \ref{defofap}.
\end{remark}


The following example shows that the bound $\bm\Lambda_1$ in \eqref{condg} is actually sharp for circular domains.

\begin{example}[Instability of a non-circular flow in a disk]\label{examins}
  Let $D$ be the unit disk centered at the origin.
Then
\begin{equation}\label{dsk1}
  \bm\Lambda_1=\lambda_2=j_{1,1}^2,\quad
\mathbf E_1={\rm span}\left\{J_0(j_{1,1}r), \, J_1(j_{1,1}r)\sin\theta,\, J_1(j_{1,1}r)\cos\theta\right\},
\end{equation}
where $(r,\theta)$ denotes the polar coordinates,  $J_n$, $ n=0,1,2,\dots$, is the Bessel function of  order $n$ of the first kind, and $j_{1,1}\approx 3.831706$ is the first positive zero of $J_1$; see  \cite[Propositions 3.4]{GL} or \cite[Lemma 2.5]{W24j}.
Consider the steady flow  in $D$ with stream function $\psi^s$ and vorticity $\omega^s$:
\begin{equation}\label{steadyflow}
\psi^s(r,\theta)=\frac{1}{j^2_{1,1}}J_1(j_{1,1} r)\cos\theta,\quad \omega^s(r,\theta)=J_1(j_{1,1} r)\cos\theta.
\end{equation}
Then  $\psi^s$ satisfies
\[-\Delta\psi^s=j_{1,1}^2\psi^s\ {\rm in}\ D,\quad \psi^s|_{\partial D}=0,\quad \int_D\psi^s d\mathbf x=0.\]
One can verify that for any
$n\in\mathbb N_+$,
\[\omega_n(t,r,\theta):=J_1(j_{1,1} r)\cos\left(\theta-\frac{1}{n} t\right)+\frac{2}{n}\]
is a solution to the Euler equation. It is clear that
\[\lim_{n\to\infty}\|\omega_n(0,\cdot)-\omega^s\|_{L^p(D)}= 0,\]
but
\[  \sup_{t\in\mathbb R, \, n\in\mathbb N_+}\|\omega_n(t,\cdot)-\omega^s\|_{L^p(D)}>0.
\]
This shows that the steady flow \eqref{steadyflow} is unstable (indeed, any non-circular steady flow in a disk is unstable; see \cite[Section~6]{W24j}).
\end{example}

Although instability may occur when $g'(\psi^s)$ reaches $\bm\Lambda_1,$ our next theorem indicates that a certain form of structural stability still holds.

\begin{theorem}[Structural stability: $g'\leq \bm\Lambda_1$]\label{thm2}
Suppose that
$\psi^s\in C^{2}(\bar D)$ solves \eqref{semilieq}, where $g\in C^{1}(\mathbb R)$. Suppose that
\begin{equation}\label{condg2}
 0\leq \min_{\bar D} g' (\psi^s)\leq \max_{\bar D}g' (\psi^s)\leq \bm\Lambda_1.
 \end{equation}
Denote by $(\omega^s,\bm\gamma^s)$ the vorticity-circulation pair related to $\psi^s$ as in \eqref{vcp01}.
Then, for any $\varepsilon>0$, there exists some $\delta>0$ such that for any smooth Euler flow
 with vorticity-circulation pair $(\omega,\bm\gamma)$, if
 \[\|\omega(0,\cdot)-\omega^s\|_{L^p(D)}+|\bm\gamma-\bm\gamma^s|<\delta,\]
 then for any $t\in\mathbb R,$ there exists some $\tilde\omega\in (\omega^s+\mathbf E_1)\cap \mathcal R_{\omega^s}$
  such that
 \[ \|\omega(t,\cdot)-\tilde\omega\|_{L^p(D)}<\varepsilon.\]

 \end{theorem}

\begin{remark}\label{vpfor2}
The conclusion of Theorem \ref{thm2} can also be stated in terms of the stream function as follows:  {
for any $\varepsilon>0$, there exists some $\delta>0,$ such that for any  smooth Euler flow with stream function $\psi$, if
 \[\|\psi(0,\cdot)-\psi^s\|_{W^{2,p}(D)}<\delta,\]
 then for any $t\in\mathbb R$, there exists some $\tilde\psi$ satisfying
 \[\tilde\psi-\psi^s\in\mathbf E_1,\quad -\Delta\tilde\psi\sim-\Delta\psi^s,\]
 such that
 \[  \|\psi(t,\cdot)-\tilde\psi\|_{W^{2,p}(D)}<\varepsilon.\]
 }

\end{remark}

 \begin{remark}\label{vpfor4}
 Suppose that $\omega^s\in\mathcal S\subset (\omega^s+\mathbf E_1)\cap \mathcal R_{\omega^s}$, and that $\mathcal S$ is isolated in $(\omega^s+\mathbf E_1)\cap \mathcal R_{\omega^s}$, i.e.,
 \[\min\left\{\|v_1-v_2\|_{L^p(D)} \;\middle|\;  v_1\in\mathcal S,\, v_2\in \left((\omega^s+\mathbf E_1)\cap \mathcal R_{\omega^s}\right)\setminus\mathcal S\right\}>0.\]
 Then by a standard continuity argument, the conclusion of Theorem \ref{thm2} can  be improved  as follows:  {for any $\varepsilon>0$, there exists some $\delta>0,$ such that for any smooth Euler flow
 with vorticity-circulation pair $(\omega,\bm\gamma)$, if
 \[\|\omega(0,\cdot)-\omega^s\|_{L^p(D)}+|\bm\gamma-\bm\gamma^s|<\delta,\]
 then for any $t\in\mathbb R,$ there exists some $\tilde\omega\in \mathcal S$
  such that
 \[ \|\omega(t,\cdot)-\tilde\omega\|_{L^p(D)}<\varepsilon.\]
 }

 \end{remark}

  An interesting problem is to study the structure of the set
$(\omega^s+\mathbf E_1)\cap \mathcal R_{\omega^s}$.
 Suppose that $\mathbf E_1$ has an orthogonal (with respect to $L^2$-norm) basis $\{\Phi_1,\dots,\Phi_d\}.$
 Denote
 \[K:=\left\{(\alpha_1,\dots,\alpha_d)\in\mathbb R^d \;\middle|\; \left(\omega^s+\sum_{i=1}^d\alpha_i\Phi_i\right)\sim \omega^s\right\}\]
Then the problem reduces to studying the set $K$. Since $(\omega^s+\mathbf E_1)\cap \mathcal R_{\omega^s}$ is obviously compact in $L^p(D)$, we deduce that $K$ is compact in $\mathbb R^d$. To obtain more information, notice that  any $(\alpha_1,\dots,\alpha_d)\in K$ satisfies
\begin{equation}\label{intoff}
 \int_DF\left(\omega^s+\sum_{i=1}^d\alpha_i \Phi_i \right) d\mathbf x=\int_DF(\omega^s) d\mathbf x,
 \end{equation}
where $F$ can be any continuous function.
 By taking various functions $F$, it is theoretically possible to determine the coefficients $\alpha_1,\dots,\alpha_d$.

When the domain $D$ has certain symmetries, the Euler equations often possess additional conserved quantities, which can be very helpful in analyzing stability; see \cite{CG,CWZ,GS,WZ23}.
Below, we consider the case where $D$ is the unit  disk centered at the origin. In this case, the moment of inertia  of the vorticity $I$ is conserved,
\begin{equation}\label{defofj}
I(\omega):=\int_D|\mathbf x|^2\omega d\mathbf x.
\end{equation}
  With the help of $I$ and the rotational symmetry of the disk, we can improve Theorems \ref{thm1} and \ref{thm2} as follows.

\begin{theorem}[Rigidity and orbital stability in a disk]\label{thm3}
Let $D$ be the unit disk centered at the origin. Suppose that $\psi^s\in C^2(\bar D)$ satisfies
 \[
\begin{cases}
-\Delta  \psi^s=g(\psi^s)\quad \mbox{in }\, D,\\
\psi^s \mbox{ is  constant on } \partial D,
\end{cases}
\]
where $g\in C^1(\mathbb R)$. Denote $\omega^s=-\Delta\psi^s.$
 Fix $1<p<\infty$.
  \begin{itemize}
  \item [(i)]  If
\[
 0\leq \min_{\bar D} g' (\psi^s)\leq \max_{\bar D}g' (\psi^s)<j_{1,1}^2,
\]
where $j_{1,1}$ is given in Example \ref{examins}, then $\omega^s$ is radial, and the associated steady flow  is stable in the following sense: for any $\varepsilon>0$, there exists some $\delta>0,$ such that for any smooth Euler flow with vorticity $\omega$, it holds that
 \[\|\omega(0,\cdot)-\omega^s\|_{L^p(D)}<\delta\quad\Longrightarrow\quad \|\omega(t,\cdot)-\omega^s\|_{L^p(D)}<\varepsilon\quad \forall\, t\in\mathbb R.\]
   \item [(ii)]  If
\[
 0\leq \min_{\bar D} g' (\psi^s)\leq \max_{\bar D}g' (\psi^s)\leq j_{1,1}^2,
\]
then $\omega^s$ can be decomposed as
\[\omega^s=\omega^r+\omega^e,\]
where $\omega^r$ is radial  and $\omega^e\in\mathbf E_1,$ and  the associated flow is orbitally stable up to rigid rotation, i.e., for any $\varepsilon>0$, there exists some $\delta>0,$ such that for any smooth Euler flow with vorticity $\omega$, it holds that
 \[\|\omega(0,\cdot)-\omega^s\|_{L^p(D)}<\delta\quad\Longrightarrow\quad \min_{v\in\mathcal O_{\omega^s}}\|\omega(t,\cdot)-v\|_{L^p(D)}<\varepsilon\quad \forall\, t\in\mathbb R,\]
 where $\mathcal O_{\omega^s}$ denotes the rotational orbit of $\omega^s$, i.e.,
 \[\mathcal O_{\omega^s}:=\left\{\omega^s(r,\theta+\alpha)\;\middle|\;  \alpha\in\mathbb R\right\}.\]
\end{itemize}
\end{theorem}

The above theorem extends Theorem 1.2 in \cite{W24j}, which asserts that any eigenstate in $\mathbf{E}_1$, as given in \eqref{dsk1}, is orbitally stable up to rigid rotation. The proof in \cite{W24j} is achieved by studying the equimeasurable partition of $\mathbf{E}_1$, which is rather technical. In this paper, by using the conservation of $I$, the analysis is significantly simplified.

This paper is organized as follows. In Section \ref{Sec2}, we investigate an elliptic problem related to the vorticity-circulation formulation of the Euler equations. In Section \ref{Sec3}, we study the Laplacian eigenvalue problem \eqref{lepzm} and give the proof of Theorem \ref{thm0}. In Section \ref{Sec4}, we prove a Burton-type stability criterion that is used in subsequent sections. Section \ref{Sec5} is devoted to the proofs of Theorems \ref{thm1} and \ref{thm2}, and Section \ref{Sec6} is devoted to the proof of Theorem \ref{thm3}.

\section{An elliptic problem}\label{Sec2}
 The aim of this section is to prove the following proposition.
\begin{proposition}\label{prop21}
For any  $v\in L^p(D)$ and any $\bm\gamma=(\gamma_1,\dots,\gamma_N)\in\mathbb R^{N}$, there is a unique solution $u\in W^{2,p}(D)$ to the following elliptic problem:
\begin{equation}\label{psieq9i}
\begin{cases}
-\Delta u=v,&\mathbf x\in D,\\
u|_{\Gamma_{i}}={\rm constant},&i= 0,1,\dots,N,\\
\int_{\Gamma_{i}}\nabla u\cdot \mathbf n dS=-\gamma_i,&i= 1,\dots,N,\\
\int_Du d\mathbf x=0.
\end{cases}
 \end{equation}
 Moreover, there exist two positive constants  $c,C$, depending only on $p$ and $D$, such that
 \begin{equation}\label{cces}
 c\|u\|_{W^{2,p}(D)}\leq \|v\|_{L^p(D)}+|\bm\gamma|\leq C\|u\|_{W^{2,p}(D)},
 \end{equation}
where $|\bm\gamma|:=\sqrt{\sum_{i=1}^N\gamma_i^2}$.

\end{proposition}

To prove Proposition \ref{prop21}, we recall some notation and facts introduced in \cite{WZ22}.
Let  $\mathsf G$ be the inverse of $-\Delta$ subject to the zero Dirichlet boundary condition,
i.e., $u:=\mathsf Gv$ is the unique solution to
\begin{equation}\label{dfgop}
\begin{cases}
-\Delta u=v,&\mathbf x\in D,\\
 u|_{\partial D}=0.
 \end{cases}
 \end{equation}
It is well known that $\mathsf G$ is a bounded, linear operator mapping $L^p(D)$ into $W^{2,p}(D)$.  Let  $\zeta_i,$ $i=0,1,\dots,N,$ be the  harmonic function in $D$ determined by
\begin{equation}\label{zeteq}
 \Delta \zeta_{i}=0 \,\,\,\mbox{in}\,\,\, D,\quad
\zeta_{i} =
\begin{cases}
1\,\,\,\mbox{on }\Gamma_i,\\
0\,\,\,\mbox{on }\partial D\setminus \Gamma_i.
\end{cases}
\end{equation}
Denote
\begin{equation}\label{pijdef}
p_{ij}:=\int_{D}\nabla\zeta_{i}\cdot\nabla\zeta_{j}dx,\quad 1\leq i,j\leq N.
\end{equation}
One can check that the matrix $(p_{ij})_{1\leq i,j\leq N}$ is symmetric and positive definite.
Let $(q_{ij})_{1\leq i,j\leq N}$ be the inverse of  $(p_{ij})_{1\leq i,j\leq N}.$
 Define
\begin{equation}\label{pvha1}
 \mathsf Pv:=\mathsf Gv+\sum_{i,j=1}^{N} q_{ij}\left(\int_{D}v\zeta_{i} d\mathbf x\right)\zeta_j,\quad h_{\bm\gamma}:= -\sum_{i,j=1}^{N}q_{ij}\gamma_{i}\zeta_{j}.
\end{equation}
Note that only $\zeta_0$ will not be used until the proof of Lemma \ref{lemmp1}.

\begin{lemma}[{\cite[Section 2]{WZ22}}]\label{lm20}
\begin{itemize}
\item[(i)] $\mathsf P$ is a bounded, linear operator mapping  $L^{p}(D)$ into $W^{2,p}(D).$
\item[(ii)] $\mathsf P$ is symmetric and positive definite,  i.e.,
   \begin{equation}\label{psym}
  \int_{D}v_{1}\mathsf Pv_{2}d\mathbf x=\int_{D}v_{2}\mathsf Pv_{1}d\mathbf x\quad\forall\, v_1,v_2\in L^p(D),
 \end{equation}
  and
    \begin{equation}\label{ppdefy}
  \int_{D}v\mathsf Pv d\mathbf x\geq 0\quad \forall\, v\in L^p(D),
  \end{equation}
  and \eqref{ppdefy} is an equality if and only if $v=0.$
\item[(iii)] Fix $v\in L^p(D)$ and $\bm\gamma=(\gamma_1,\dots,\gamma_N)\in\mathbb R^N.$ Then $ u:=\mathsf Pv+h_{\bm\gamma}$ is the unique solution in $W^{2,p}(D)$ to the following elliptic problem:
\begin{equation}\label{psieq9}
\begin{cases}
-\Delta   u=v,&\mathbf x\in D,\\
 u|_{\Gamma_0}=0,\\
 u|_{\Gamma_{i}}=c_i,&i= 1,\dots,N,\\
\int_{\Gamma_{i}}\nabla u\cdot \mathbf n dS=-\gamma_i,&i= 1,\dots,N.
\end{cases}
 \end{equation}
\end{itemize}
\end{lemma}

Now we are ready to give the proof of Proposition \ref{prop21}.  For convenience, we use \(\mathsf A_v\) to denote the integral average of \(v \in L^1(D)\) over \(D\), i.e.,
\[
\mathsf A_v = \frac{1}{|D|} \int_D v  d\mathbf x.
\]

\begin{proof}[Proof of Proposition \ref{prop21}]
By items (i) and (iii) of Lemma \ref{lm20}, it is easy to see that
\[u:=\mathsf Pv+h_{\bm\gamma}-\mathsf A_{\mathsf Pv+h_{\bm\gamma}},\]
solves \eqref{psieq9i} and satisfies the estimate \eqref{cces}. On the other hand,   the elliptic problem \eqref{psieq9i} has at most one solution in $W^{2,p}(D)$ by integration by parts.
\end{proof}

\section{Proof of Theorem \ref{thm0}}\label{Sec3}

We begin by reformulating \eqref{lepzm} as an operator eigenvalue problem. Define $\mathsf T:\mathring L^2(D)\mapsto \mathring L^2(D)$ by setting
 \[\mathsf Tv=\mathsf Pv-\mathsf A_{\mathsf Pv},\quad v\in\mathring L^2(D),\]
 where $\mathsf P$ is given by \eqref{pvha1}, and
 \[\mathring L^2(D):=\left\{v\in L^2(D) \;\middle|\; \mathsf A_v=0\right\}.\]
By Lemma \ref{lm20}(iii),  for any $v\in\mathring L^2(D)$,  $u:=\mathsf Tv$ is the unique solution to
\begin{equation}\label{psieq10}
\begin{cases}
-\Delta u=v,&\mathbf x\in D,\\
u\in\mathbf Y.
\end{cases}
 \end{equation}
 In other words, $\mathsf T$ is a bijective mapping $\mathring L^2(D)$ onto $\mathbf Y.$
Hence, we can rewrite \eqref{lepzm} as
  \begin{equation}\label{oepbm}
  v=\bm\Lambda \mathsf Tv,\quad v\in\mathring L^2(D).
  \end{equation}

 \begin{proof}[Proof of Theorem \ref{thm0}(i)(ii)]
 In view of Lemma \ref{lm20}, $\mathcal{T}$ is a compact, symmetric, and positive-definite operator mapping $\mathring{L}^2(D)$ into itself.
Item~(i) then follows directly from the Hilbert--Schmidt theorem.
Note that the positive-definiteness of $\mathcal{T}$ ensures that $\bm{\Lambda}_1 > 0$.
Item~(ii) can be proved by the method of eigenfunction expansion.
Since the argument is quite standard (see, for example, \cite[Section~6.5]{LCE}), we omit the detailed proof.
 \end{proof}

\begin{remark}\label{rk31}
Since $\mathsf T:\mathring L^2(D)\mapsto \mathbf Y$ is bijective, Theorem \ref{thm0}(ii) can also be  stated  as follows:  \emph{for any $v\in\mathring L^2(D)$, it holds that
\begin{equation}\label{rk31in}
\bm\Lambda_1\int_D(\mathsf Tv)^2d\mathbf x\leq \int_D |\nabla \mathsf T v|^2d\mathbf x,
\end{equation}
and the  equality holds if and only if $v\in\mathbf E_1.$}
\end{remark}

\begin{remark}
Theorem \ref{thm0}(ii)  implies another inequality:  \emph{for any $u\in\mathbf Y$, it holds that
\begin{equation}\label{ineq212}
\int_D |\nabla u|^2d\mathbf x\leq \frac{1}{\bm\Lambda_1} \int_D(\Delta u)^2d\mathbf  x,
\end{equation}
and the  equality holds if and only if $u\in\mathbf E_1.$}
To see this,  denote $v=-\Delta u$. Then $u=\mathsf Tv$. By a straightforward computation,
\begin{align*}
      \frac{1}{\bm\Lambda_1}\int_D (\Delta u)^2 d\mathbf x -\int_D|\nabla u|^2d\mathbf x&=
      \frac{1}{\bm\Lambda_1}\int_D v^2 d\mathbf x -\int_Duv d\mathbf x\\
      &\geq  2\int_D uv d \mathbf x-\bm\Lambda_1\int_Du^2 d\mathbf x -\int_Duv d\mathbf x\\
     &=  \int_Duv d\mathbf x-\bm\Lambda_1\int_Du^2 d\mathbf x \\
     &= \int_D|\nabla u|^2d\mathbf x-\bm\Lambda_1\int_Du^2 d\mathbf x\\
     &\geq 0.
   \end{align*}
Moreover,  the first inequality is an equality if and only if $-\Delta u=\bm\Lambda_1 u$, or equivalently, $v\in\mathbf E_1;$ and the last inequality is an equality if and only if  $u\in\mathbf  E_1.$
Note that \eqref{ineq212} plays an essential role in \cite{W24j}.
\end{remark}

The rest of this section is devoted to the proof of Theorem \ref{thm0}(iii). To begin with, we recall some basic properties of the constant $\Lambda_1$ which have been proved in \cite{WZ22}.

 \begin{lemma}[{\cite[Section 3]{WZ22}}]\label{lemmp0}
 It holds that
 \begin{equation}\label{infmm}
 \Lambda_1=\inf_{u\in\mathcal X,\,u\not\equiv 0} \frac{\int_D|\nabla u|^2 d\mathbf x}{\int_Du^2 d\mathbf x},
 \end{equation}
 where
 \[ \mathcal X:=\left\{u\in H^1(D) \;\middle|\;  u|_{\Gamma_0}=0,\,u|_{\Gamma_i}=c_i,\,i=1,\dots, N\right\}.\]
 Moreover, the infimum in \eqref{infmm} is attained; and if $u$ is a minimizer, then
  \begin{itemize}
  \item[(i)]$u\in\mathcal Y$ and satisfies $-\Delta u=\Lambda_1u,$ and,
  \item[(ii)] either $u>0$ in $D$, or else $u<0$ in $D$.
  \end{itemize}
 \end{lemma}

Next, we study the constant $\bm{\Lambda}_1$ by
considering the following minimization problem:
\begin{equation}\label{vpbfm}
\mathbf{m} := \inf_{u \in \mathbf{X},\, u \not\equiv 0}
\frac{\int_D |\nabla u|^2  d\mathbf{x}}{\int_D u^2 d\mathbf{x}},
\end{equation}
where
\[
\mathbf{X} := \left\{ u \in H^1(D) \;\middle|\; \mathsf A_u = 0,\
u|_{\Gamma_i} = c_i,\ i = 0,1,\dots, N \right\}.
\]
It is clear that $\mathbf Y\subset\mathbf X.$

  \begin{lemma}\label{lemmp1}
  $\mathbf m=\bm\Lambda_1,$ and the set of minimizers of \eqref{vpbfm} is exactly $\mathbf E_1.$
\end{lemma}

    \begin{proof}
 First, we show that the infimum is attained. It suffices to notice that any minimizing sequence, with unit $L^2$-norm after renormalization, is bounded in $H^1(D)$, and that $\mathbf X$ is closed, and thus weakly closed in $H^1(D)$.

   Next, we show that any minimizer $u$ belongs to $\mathbf Y$, and satisfies $-\Delta u=\mathbf m u$ in $D$.
In fact, for any $\phi\in \mathbf X,$ it holds that $u+s\phi\in\mathbf X$ for any $s\in\mathbb R$, and $u+s\phi\not\equiv 0$ if $|s|$ is sufficiently small. Thus
\[\frac{d}{ds}\left(\frac{\int_D|\nabla (u+s\phi)|^2d\mathbf x}{\int_D (u+s\phi)^2d\mathbf x}\right)\bigg|_{s=0}=0,\]
which yields
\begin{equation}\label{wfeq1}
\int_D\nabla u\cdot\nabla \phi d\mathbf x= \mathbf m\int_D u \phi d\mathbf x.
\end{equation}
Taking
$\phi=\varphi-\mathsf A_\varphi$ in \eqref{wfeq1},
where $\varphi\in C_c^\infty(D)$ is arbitrary,  and using $\mathsf A_u=0$, we obtain
\begin{equation}\label{wfeq2}
\int_D\nabla u\cdot\nabla \varphi d\mathbf x=\mathbf m\int_D u \varphi d\mathbf x\quad\forall\, \varphi\in C_c^\infty(D),
\end{equation}
which implies that $-\Delta  u=\mathbf m  u$ in $D$.
Taking in \eqref{wfeq1}
\[\phi=\zeta_i-\mathsf A_{\zeta_i},\quad i=0,  1,\dots,N,
\]
where $\zeta_i$ is given by  \eqref{zeteq},
we obtain
\[\int_D\nabla  u\cdot\nabla \zeta_i d\mathbf x= \mathbf m\int_D u \zeta_id\mathbf x \quad\forall\, i=0,  1,\dots,N, \]
which, after integrating by parts, yields
\[\int_{\Gamma_i}\frac{\partial u}{\partial \mathbf n}dS= 0 \quad\forall \, i=0,  1,\dots,N.\]
Thus $u\in\mathbf Y.$

Now we are ready to prove the desired result.  From the above discussion, we know that
\[\begin{cases}
-\Delta u=\mathbf mu,&\mathbf x\in D,\\
u\in\mathbf Y
\end{cases}
\]
has a nontrivial solution, which implies
$\bm\Lambda_1\leq \mathbf m.$ On the other hand, by choosing $u \in \mathbf{E}_1$, $u \not\equiv 0$, we have
\[
\mathbf{m} \leq \frac{\int_D |\nabla u|^2 d\mathbf{x}}{\int_D u^2 d\mathbf{x}} = \bm\Lambda_1.
\]
Hence we have verified $\mathbf{m} =\bm\Lambda_1$.
The assertion that the set of minimizers of \eqref{vpbfm} is exactly $\mathbf{E}_1$ is then straightforward.
  \end{proof}

 \begin{remark}
A straightforward corollary of Lemma \ref{lemmp1} is as follows:  \emph{for any $u\in\mathbf X$, it holds that
 \begin{equation}\label{inforx}
\bm\Lambda_1\int_Du^2d\mathbf x\leq \int_D|\nabla u|^2d\mathbf x,
\end{equation}
and the equality holds if and only if $u\in\mathbf E_1$.} Obviously, this is a stronger conclusion than Theorem \ref{thm0}(ii).
 \end{remark}

  Now we are ready to give the proof of Theorem \ref{thm0}(iii).

 \begin{proof}[Proof of Theorem \ref{thm0}(iii)]
   Fix $u\in \mathbf E_1$ satisfying $u\not\equiv 0$. By Lemma \ref{lemmp1}, we have that
   \begin{equation}\label{bj1}
   \bm\Lambda_1=\frac{\int_D|\nabla u|^2 d\mathbf x}{\int_Du^2 d\mathbf x}.
   \end{equation}
Notice that $u-u|_{\Gamma_0}\in\mathcal X$, and  $u\not\equiv u|_{\Gamma_0}$ (since otherwise $u\equiv 0$). Therefore, according to \eqref{infmm}, we have that
   \begin{equation}\label{bj2}
    \Lambda_1\leq \frac{\int_D|\nabla (u-u|_{\Gamma_0})|^2 d\mathbf x}{\int_D(u-u|_{\Gamma_0})^2 d\mathbf x}=\frac{\int_D|\nabla u|^2 d\mathbf x}{\int_Du^2 d\mathbf x+\int_D(u|_{\Gamma_0})^2 d\mathbf x}\leq \frac{\int_D|\nabla u|^2 d\mathbf x}{\int_Du^2 d\mathbf x}=\bm\Lambda_1,
    \end{equation}
   where we have used $\mathsf A_u=0$.
   Our aim is to prove $\bm\Lambda_1>\Lambda_1.$   Suppose otherwise, i.e., $\bm\Lambda_1=\Lambda_1;$ then the inequalities in \eqref{bj2} are both equalities. In particular, $u$ attains the infimum in \eqref{infmm}.  By Lemma \ref{lemmp0}, we deduce that $u$ has a constant sign in $D$, which contradicts  $\mathsf A_u=0$.
 \end{proof}

\section{A Burton-type stability criterion}\label{Sec4}

Recall that the kinetic energy $E$ of the  fluid is  given by
\[E=\frac{1}{2}\int_D|\mathbf v|^2 d\mathbf x=\frac{1}{2}\int_D|\nabla \psi|^2 d\mathbf x.\]
From now on,  we regard $E$ as a functional of the vorticity-circulation pair $(\omega,\bm\gamma)$.

\begin{lemma}[{\cite[Section 2]{WZ22}}]\label{lm41}
In terms of the vorticity-circulation pair $(\omega,\bm\gamma)$,   $E$ can be expressed  as
\begin{equation}\label{eexp1}
E(\omega,\bm\gamma)= \frac{1}{2}\int_{D}\omega\mathsf P\omega d\mathbf x+\int_{D}h_{\bm\gamma}\omega d\mathbf x+\frac{1}{2}\sum_{i,j=1}^{N}q_{ij}\gamma_{i}\gamma_{j}.
\end{equation}
Moreover,  $E$ is well defined in $L^p(D)\times \mathbb R^N$, and satisfies the following properties:
\begin{itemize}
\item[(i)] If $v_{n}\rightharpoonup v$ in $L^{p}(D)$ and $\bm\gamma_{n}\to\bm\gamma$ in $\mathbb R^{N},$ then
\[\lim_{n\to \infty}E(v_{n},\bm\gamma_{n})=E(v,\bm\gamma),\]
where $``\rightharpoonup"$ denotes weak convergence in $L^p(D)$.
\item[(ii)] For any  bounded set $K$ of $L^{p}(D)$  and any bounded set  $I$ of $\mathbb R^{N}$, the functional $E$ is uniformly continuous on $K\times I.$
\end{itemize}
\end{lemma}

Next, we introduce the concept of admissible pair as an abstraction of the vorticity-circulation pair of the Euler equations.
\begin{definition}[Admissible pair]\label{defadms}\label{defofap}
Let $\zeta:\mathbb R\mapsto  L^p(D)$ be a map, and let $\mathbf b\in\mathbb R^N$. If for any $t\in\mathbb R$, it holds that
 \[E(\zeta(t),\mathbf b)=E(\zeta(0),\mathbf b),\quad \zeta(t)\in \mathcal R_{\zeta(0)},\]
then   $(\zeta,\mathbf b)$ is called an admissible pair. If $\zeta$ is additionally continuous, then $(\zeta,\mathbf b)$ is called a continuous admissible pair.
\end{definition}

By the conservation laws introduced in Section \ref{Sec1}, we know that the vorticity-circulation pair $(\omega(t,\cdot),\bm\gamma)$ for any sufficiently smooth fluid flow is an admissible pair.

The following Burton-type stability criterion, inspired mainly by Burton's original work \cite{B05} (see also \cite{WZ22, WMA}), is essential for proving Theorems \ref{thm1} and \ref{thm2}.

\begin{theorem}\label{thmbsc}
Let $\bm\gamma\in\mathbb R^N$ be fixed, and let  $\mathcal R$ be the rearrangement class of some  $L^p$-function in $D$.
Denote by $\mathcal M$ the set of maximizers of $E(\cdot,\bm\gamma)$ relative to $\mathcal R,$ i.e.,
\[\mathcal M=\left\{v\in\mathcal R \;\middle|\;  E(v,\bm\gamma )=M\right\},\quad M:=\sup_{v\in \mathcal R}E(v,\bm\gamma).\]
 Then $\mathcal M$ is stable in the following sense: for any $\varepsilon>0,$ there exists some $\delta>0$, such that for any admissible pair $(\zeta,\mathbf b)$ in the sense of Definition \ref{defofap}, if
\[ \|\zeta(0)-\hat\omega\|_{L^p(D)}+|\mathbf b-\bm\gamma|<\delta\]
for some $\hat\omega\in\mathcal M,$  then for any $t\in\mathbb R$, there exists some $\tilde\omega\in\mathcal M$ such that
\[ \|\zeta(t)-\tilde\omega\|_{L^p(D)}<\varepsilon.\]
\end{theorem}

 \begin{proof}

We begin by proving the following compactness property:
  \begin{equation}\label{claim}
  \textit{Any maximizing sequence of $E(\cdot,\bm\gamma)$ relative to $\mathcal R$ is compact in $L^p(D)$.}
  \end{equation}
More precisely,
 for any $\{v_n\}_{n=1}^\infty\subset\mathcal R$ satisfying
 \[\lim_{n\to\infty} E(v_n,\bm\gamma)=M,\quad M:=\sup_{v\in\mathcal R}E(v,\bm\gamma),\]
 there exists a subsequence, still denoted by $\{v_n\}$ for simplicity, and some $\tilde v\in\mathcal R$, such that $v_n$ converges to $\tilde v$  {strongly} in $L^p(D)$ as $n\to\infty$.
 In fact, since $\{v_n\}$ is obvious bounded in $L^p(D)$, we may assume,
up to a subsequence, that
 $v_n$ converges to some $ \tilde v\in \overline{\mathcal R^w}$  {weakly}  in $L^p(D)$,
where $\overline{\mathcal R^w}$ denotes the weak closure of $\mathcal R$ in $L^p(D).$
In view of Lemma \ref{lm41}(i), we have that $E(\tilde v,\bm\gamma)=M.$
Since
 \[\sup_{v\in\mathcal R}E(v,\bm\gamma)=\sup_{v\in\overline{\mathcal R^w}}E(v,\bm\gamma),\]
 we obtain
 \begin{equation}\label{htm02}
 E(\tilde v,\bm\gamma)\geq  E(v,\bm\gamma)\quad \forall \, v\in \overline{\mathcal R^w}.
 \end{equation}
To proceed, observe the following identity:
 \begin{equation}\label{htm03}
 E(v,\bm\gamma)- E( \tilde v,\bm\gamma)=L(v)-L(\tilde v)+\frac{1}{2}\int_D(v-\tilde v)\mathsf P(v-\tilde v)d\mathbf x,
 \end{equation}
where
\[L(w):=\int_{D}(\mathsf P\tilde v+h_{\bm\gamma})wd\mathbf x.\]
From \eqref{htm02}, \eqref{htm03}, and the positive definiteness of $\mathsf P$ (see \eqref{ppdefy}),  we obtain
 \begin{equation}\label{lleq}
  L(\tilde v)\geq L(v) \quad \forall\, v\in \overline{\mathcal R^w},
  \end{equation}
  and the equality holds if and only if $v=\tilde v$. In other words, $\tilde v$ is the unique maximizer of $L$ relative to  $\overline{\mathcal R^w}$. On the other hand, by applying \cite[Theorem 4]{B87}, there exists some $\hat v\in {\mathcal R}$ such that
  \[L(\hat v)=\sup_{v\in\overline{\mathcal R^w}}L(v).\]
So  we obtain $\tilde v=\hat v\in  \mathcal R$. Strong convergence then follows from the uniform convexity of $L^p$ spaces for $1<p<\infty$.

With \eqref{claim} in hand, we are ready to prove stability. It suffices to show that for any sequence of admissible pairs $\{(\zeta_n,\mathbf b_n)\}$ and any sequence of moments $\{t_n\}\subset\mathbb R$, if
\begin{equation}\label{ems0}
 \|\zeta(0)-\hat \omega\|_{L^p(D)}+|\mathbf b_n-\bm\gamma|\to 0
 \end{equation}
for some $\hat \omega\in\mathcal M$ as $n\to\infty,$  then, up to a subsequence,
\begin{equation}\label{ems1}
\|\zeta_{n}(t_n)-\tilde \omega\|_{L^p(D)}\to 0
\end{equation}
 for some $\tilde \omega\in\mathcal M$ as $n\to\infty.$
To this end, note that by Lemma \ref{lm41}(i) and \eqref{ems0},
  \[
  \lim_{n\to \infty} E(\zeta_n(0),\mathbf b_n)= E(\hat \omega,\bm\gamma)=M,
  \]
which implies, using that $(\zeta_n,\mathbf b_n)$ is an admissible pair,
\[\lim_{n\to \infty}  E(\zeta_n(t_n),\mathbf b_n) =M.\]
By Lemma \ref{lm41}(ii), we further obtain
    \begin{equation}\label{cvrm1}
\lim_{n\to \infty}  E(\zeta_n(t_n),\bm\gamma ) =M.
  \end{equation}
  To apply \eqref{claim}, we introduce a ``follower'' $\xi_n$ associated with each $\zeta_n(t_n)$. More specifically, by \cite[Lemma 2.3]{BACT}, for each $n$ there exists $\xi_n \in \mathcal{R}$ such that
\begin{equation}\label{infvw}
\|\xi_n-\zeta_n(t_n)\|_{L^p(D)}= \inf\left\{\|v-w\|_{L^p(D)}\;\middle|\;  v\in\mathcal R,\ w\in\mathcal R_{\zeta_n(t_n)}\right\}.
\end{equation}
In particular,
\begin{equation}\label{eta1}
\|\xi_n-\zeta_n(t_n)\|_{L^p(D)}\leq \|\zeta_n(0)-\hat\omega  \|_{L^p(D)},
\end{equation}
which in combination with \eqref{ems0} implies
\begin{equation}\label{eta2}
\lim_{n\to \infty}\|\xi_n-\zeta_n(t_n)\|_{L^p(D)}=0.
\end{equation}
From  \eqref{cvrm1} and \eqref{eta2}, we can apply  Lemma \ref{lm41}(ii) again to obtain
\begin{equation}\label{eta3}
\lim_{n\to \infty} E(\xi_n,\bm\gamma)=M.
\end{equation}
To conclude, we have constructed a  sequence $\{\xi_n\}\subset\mathcal R$ such that
\eqref{eta3} holds. Then we can apply \eqref{claim} to find $\tilde \omega\in\mathcal M$ such that, up to a subsequence,
\begin{equation}\label{eta4}
\lim_{n\to \infty}\|\xi_n-\tilde \omega\|_{L^p(D)}=0.
\end{equation}
The desired assertion \eqref{ems1} then
follows from \eqref{eta2} and \eqref{eta4}. The proof is complete.

 \end{proof}

\begin{remark}
According to  \eqref{claim},  $\mathcal{M}$ is a nonempty, compact subset of $L^p(D)$.
\end{remark}

\begin{remark}
If  $\mathcal N\subset\mathcal M$ is isolated, i.e.,
\[\inf_{v\in\mathcal N,\,w\in\mathcal M\setminus\mathcal N}\|v-w\|_{L^p(D)}>0,\]
then $\mathcal N$ is stable in the following sense:
 for any $\varepsilon>0,$ there exists some $\delta>0$, such that for any  \emph{continuous} admissible pair $(\zeta,\mathbf b)$, if
\[ \|\zeta(0)-\hat\omega\|_{L^p(D)}+|\mathbf b-\bm\gamma|<\delta\]
for some $\hat\omega\in\mathcal N,$  then for any $t\in\mathbb R$, there exists some $\tilde\omega\in\mathcal N$ such that
\[ \|\zeta(t)-\tilde\omega\|_{L^p(D)}<\varepsilon.\]
 This follows directly from a standard continuity argument.
\end{remark}

\section{Proofs of Theorems \ref{thm1} and \ref{thm2}}\label{Sec5}

Having made enough preparations, we are ready to prove Theorems \ref{thm1} and \ref{thm2}. Let $(\omega^s,\bm\gamma^s)$ be the vorticity-circulation pair of the steady flow in Theorems \ref{thm1} or  \ref{thm2}.
Consider the following maximization problem:
\[M=\sup_{v\in\mathcal R_{\omega^s}}E(v,\bm\gamma^s).\]
Denote by $\mathcal M$ the set of maximizers of $E(\cdot,\bm\gamma^s)$ relative to $\mathcal R_{\omega^s}$.

\subsection{Proof  of Theorem \ref{thm1}}
By Theorem \ref{thmbsc},  Theorem \ref{thm1} follows immediately from the following proposition.

 \begin{proposition}\label{propsingleton}
 Under the conditions of Theorem \ref{thm1},  it holds that $\mathcal M=\{\omega^s\}.$

 \end{proposition}\label{propm1}
 \begin{proof}
 Without loss of generality, we may assume that $g\in C^1(\mathbb R)$, and satisfies
 \begin{equation}\label{aspg1}
  0\leq g'(s)\leq \bm\Lambda_1-\epsilon\quad\forall\,s\in\mathbb R
 \end{equation}
for some sufficiently small $\epsilon>0$, and
 \begin{equation}\label{aspg2}
 g' \equiv {\rm positive\,\, constant}\,\,\, {\rm in} \,\,\, \mathbb R\setminus \left[\min_{\bar D}\mathsf G\omega^s-1,\max_{\bar D}\mathsf G\omega^s+1\right].
 \end{equation}
This can be achieved by redefining the values of $g$ outside the interval $\left[\min_{\bar D}\mathsf G\omega^s,\max_{\bar D}\mathsf G\omega^s\right]$; see,  for example,  Lemma 2.4 in \cite{W22} for details. Note that \eqref{aspg1} implies that
\begin{equation}\label{texp}
G(s+\tau)\leq G(s)+ g(s)\tau+\frac{1}{2}(\bm\Lambda_1-\epsilon)\tau^2\quad\forall\,
s,\tau\in\mathbb R.
\end{equation}
Let $G$ be an antiderivative of $g$, and
let $\hat G$ be the Legendre transform  of $G$, i.e.,
\[\hat G(s):=\sup_{\tau\in\mathbb R}\left(s\tau-G(\tau)\right).\]
The condition \eqref{aspg1} ensures that   $\hat G$
 is  real-valued and locally Lipschitz; see  \cite[Lemma 2.3]{W22}.
 Moreover, from the definition of Legendre transform, we see that
\begin{equation}\label{hatg}
\hat G(s)+G(\tau)\geq s\tau\quad\forall\,s,\tau\in\mathbb R,
\end{equation}
and the  equality holds if and only if $s=g(\tau)$.
 In view of  \eqref{hatg} and the relation
$\omega^s=g(\psi^s),$
we have
\begin{equation}\label{hatg913}
\hat G(\omega^s+\varrho)+ G(\psi^s+\mathsf T\varrho) \geq (\psi^s+\mathsf T\varrho)( \omega^s+\varrho)
\end{equation}
and \begin{equation}\label{ggeq}
G(\psi^s)+\hat G( \omega^s)=\psi^s \omega^s.
\end{equation}

Fix  $\varrho\in\mathcal R_{\omega^s}-\omega^s,$ $\varrho\not\equiv 0.$  Our aim is to show that
\begin{equation}\label{aim}
E(\omega^s,\bm\gamma^s)>E(\omega^s+\varrho,\bm\gamma^s).
\end{equation}
By a direct calculation,
\begin{equation}\label{comp1}
\begin{split}
 &E(\omega^s,\bm\gamma^s)-E(\omega^s+\varrho,\bm\gamma^s)\\
 =&\frac{1}{2}\int_D \omega^s\mathsf P \omega^s d\mathbf x-\frac{1}{2}\int_D( \omega^s+\varrho)\mathsf P( \omega^s+\varrho) d\mathbf x-\int_D\varrho h_{\bm\gamma^s} d\mathbf x \\
 =&- \frac{1}{2}\int_D\varrho\mathsf P\varrho d\mathbf x- \int_D\psi^s\varrho  d\mathbf x\\
 =&- \frac{1}{2}\int_D\varrho\mathsf T\varrho d\mathbf x- \int_D\psi^s\varrho  d\mathbf x.
   \end{split}
\end{equation}
On the other hand, using  \eqref{texp}, \eqref{hatg913}, \eqref{ggeq}, and the fact that $\omega^s+\varrho\in\mathcal R_{\omega^s},$ we have the following:
\begin{equation}\label{comp2}
\begin{split}
 0=& \int_D\hat G(\omega^s+\varrho)d\mathbf x-\int_D\hat G(\omega^s)d\mathbf x\\
 =& \int_D\hat G(\omega^s+\varrho)d\mathbf x-\int_D\psi^s \omega^s-G(\psi^s)d\mathbf x\\
 \geq& \int_D ( \omega^s+\varrho)(\psi^s+\mathsf T\varrho)-G(\psi^s +\mathsf T\varrho)d\mathbf x -\int_D\psi^s \omega^s-G(\psi^s )d\mathbf x\\
 =&  - \int_DG(\psi^s+\mathsf T\varrho )-G(\psi^s)d\mathbf x +\int_D\psi^s\varrho d\mathbf x+\int_D(\omega^s+\varrho) \mathsf T\varrho d\mathbf x\\
 \geq&-\int_Dg(\psi^s) \mathsf T\varrho +\frac{1}{2}(\bm\Lambda_1-\epsilon)(\mathsf T\varrho)^2d\mathbf x+\int_D\psi^s\varrho d\mathbf x+\int_D(\omega^s+\varrho) \mathsf T\varrho d\mathbf x\\
 =&\int_D\psi^s\varrho d\mathbf x+\int_D\varrho \mathsf T\varrho d\mathbf x-\frac{1}{2}(\bm\Lambda_1-\epsilon)\int_D(\mathsf T\varrho)^2d\mathbf x.
  \end{split}
\end{equation}
Note that we have used $\omega^s=g(\psi^s)$ in the last equality.
Combining \eqref{comp1} and \eqref{comp2}, and applying the inequality \eqref{rk31in} in Remark \ref{rk31} (noting that $\varrho\in\mathring L^2(D)$), we obtain
\[
E( \omega^s,\bm\gamma^s)-E( \omega^s+\varrho,\bm\gamma^s) \geq \frac{1}{2}\int_D\varrho\mathsf T\varrho d\mathbf x  -\frac{1}{2} (\bm\Lambda_1-\epsilon)\int_D (\mathsf T\varrho)^2d\mathbf x\geq \frac{1}{2}\epsilon\int_D (\mathsf T\varrho)^2d\mathbf x.
\]
Since $\varrho\not \equiv 0,$ it follows that $\mathsf T\varrho\not \equiv 0$. Hence, the desired claim \eqref{aim} holds, and the proof is complete.

\end{proof}

\subsection{Proof of Theorem  \ref{thm2}}

 Theorem \ref{thm2} is a straightforward corollary of the following proposition.

\begin{proposition}\label{propm2}
Under the conditions of Theorem \ref{thm2}, it holds that
 \begin{equation}\label{msubset}
 \omega^s\in\mathcal M\subset (\omega^s+\mathbf E_1)\cap \mathcal R_{\omega^s}.
 \end{equation}
\end{proposition}

  \begin{proof}

  As in the proof of Proposition \ref{propm1}, we may assume that
\[
 g\in C^1(\mathbb R),\quad  0\leq g'(s)\leq \bm\Lambda_1 \quad\forall\,s\in\mathbb R,
\]
\[
 g' \equiv {\rm positive\,\, constant}\,\,\, {\rm in} \,\,\, \mathbb R\setminus \left[\min_{\bar D} \mathsf G\omega^s -1,\max_{\bar D} \mathsf G\omega^s +1\right].
\]
Then \eqref{texp} becomes
\[G(s+\tau)\leq G(s)+ g(s)\tau+\frac{1}{2} \bm\Lambda_1 \tau^2\quad\forall\,
s,\tau\in\mathbb R.\]
In this case, \eqref{comp1} still holds, i.e.,
\begin{equation}\label{comp3}
E(\omega^s,\bm\gamma^s)-E(\omega^s+\varrho,\bm\gamma^s)
 =- \frac{1}{2}\int_D\varrho\mathsf T\varrho d\mathbf x- \int_D\psi^s\varrho  d\mathbf x;
 \end{equation}
while \eqref{comp2} becomes
\begin{equation}\label{comp4}
\begin{split}
 0=& \int_D\hat G(\omega^s+\varrho)d\mathbf x-\int_D\hat G(\omega^s)d\mathbf x\\
 =& \int_D\hat G(\omega^s+\varrho)d\mathbf x-\int_D\psi^s \omega^s-G(\psi^s)d\mathbf x\\
 \geq& \int_D ( \omega^s+\varrho)(\psi^s+\mathsf T\varrho)-G(\psi^s +\mathsf T\varrho)d\mathbf x -\int_D\psi^s \omega^s-G(\psi^s )d\mathbf x\\
 =&  - \int_DG(\psi^s+\mathsf T\varrho )-G(\psi^s)d\mathbf x +\int_D\psi^s\varrho d\mathbf x+\int_D(\omega^s+\varrho) \mathsf T\varrho d\mathbf x\\
 \geq&-\int_Dg(\psi^s) \mathsf T\varrho +\frac{1}{2} \bm\Lambda_1 (\mathsf T\varrho)^2d\mathbf x+\int_D\psi^s\varrho d\mathbf x+\int_D(\omega^s+\varrho) \mathsf T\varrho d\mathbf x\\
 =&\int_D\psi^s\varrho d\mathbf x+\int_D\varrho \mathsf T\varrho d\mathbf x-\frac{1}{2} \bm\Lambda_1 \int_D(\mathsf T\varrho)^2d\mathbf x.
  \end{split}
\end{equation}
From \eqref{comp3} and \eqref{comp4}, we have that
 \begin{equation}\label{egeqe}
E( \omega^s,\bm\gamma^s)-E( \omega^s+\varrho,\bm\gamma^s) \geq \frac{1}{2}\int_D\varrho\mathsf T\varrho d\mathbf x  -\frac{1}{2} \bm\Lambda_1 \int_D (\mathsf T\varrho)^2d\mathbf x\geq  0.
\end{equation}
Hence $\omega^s\in\mathcal M$. Moreover, by Theorem \ref{thm0}(ii), if the equality holds, then $\varrho\in\mathbf E_1$. This verifies \eqref{msubset}, and the proof is complete.

 \end{proof}

 \begin{remark}
It can be proved that
 \begin{equation}\label{msubrv}
 \mathcal M=\left\{\omega^s+\varrho \;\middle|\;  \varrho\in\mathbf E_1,\ \omega^s+\varrho\in   \mathcal R_{\omega^s},\ \mathsf A_{\mathsf P\varrho}\mathsf A_{\omega^s}+\mathsf A_{\varrho h_{\bm\gamma^s}}=0\right\}.
 \end{equation}
In particular,  if $\bm\gamma^s=\mathbf 0$ or  $D$ is simply connected, and $\mathsf A_{\omega^s}=0$,
then
$\mathcal M=(\omega^s+\mathbf E_1)\cap\mathcal R_{\omega^s}.
$
 To prove \eqref{msubrv}, by Proposition \ref{propm2}, it suffices to show that for any $\varrho$ satisfying $\varrho\in\mathbf E_1$ and $\omega^s+\varrho\in\mathcal R_{\omega^s},$
 \[ E(\omega^s+\varrho,\bm\gamma^s)=E(\omega^s,\bm\gamma^s) \quad\Longleftrightarrow\quad \mathsf A_{\mathsf P\varrho}\mathsf A_{\omega^s}+\mathsf A_{\varrho h_{\bm\gamma^s}}=0\]
This is a straightforward consequence of the following observation: for any $\varrho$ satisfying $\varrho\in\mathbf E_1$ and $\omega^s+\varrho\in\mathcal R_{\omega^s},$
 \begin{align*}
 &E(\omega^s,\bm\gamma^s)-E(\omega^s+\varrho,\bm\gamma^s)\\
 =&- \frac{1}{2}\int_D\varrho\mathsf T\varrho d\mathbf x- \int_D(\mathsf P\omega^s+h_{\bm\gamma^s})\varrho  d\mathbf x\\
 =&- \frac{1}{2\bm\Lambda_1}\int_D\varrho^2 d\mathbf x- \int_D\omega^s\mathsf P\varrho d\mathbf x-\int_D\varrho h_{\bm\gamma^s}d\mathbf x\\
 =&- \frac{1}{2\bm\Lambda_1}\int_D\varrho^2 d\mathbf x- \int_D\omega^s\mathsf T\varrho d\mathbf x-\mathsf A_{\mathsf P\varrho}\int_D\omega^sd\mathbf x -\int_D\varrho h_{\bm\gamma^s}d\mathbf x\\
 =&- \frac{1}{2\bm\Lambda_1}\int_D\varrho^2 d\mathbf x- \frac{1}{\bm\Lambda_1}\int_D\omega^s \varrho d\mathbf x-\mathsf A_{\mathsf P\varrho}\int_D\omega^sd\mathbf x -\int_D\varrho h_{\bm\gamma^s}d\mathbf x\\
 =&-\mathsf A_{\mathsf P\varrho}\int_D\omega^sd\mathbf x -\int_D\varrho h_{\bm\gamma^s}d\mathbf x.
 \end{align*}
Note that in the last equality we have used
\[\int_D\omega^s\varrho  \mathbf x=-\frac{1}{2}\int_D\varrho^2 \mathbf x,\]
which follows from  $\|\omega^s+\varrho\|_{L^2(D)}=\|\omega^s\|_{L^2(D)}.$

 \end{remark}

\section{Proof of Theorem \ref{thm3}}\label{Sec6}

\begin{proof}[Proof of Theorem \ref{thm3}(i)]
By rotational symmetry, we have
\begin{equation}\label{rotainv}
E(v) = E(\omega^s) \quad \text{and} \quad v \in \mathcal R_{\omega^s}
\end{equation}
for any $v\in  \mathcal O_{\omega^s}.$
On the other hand, we have proved in Proposition \ref{propsingleton} that $\omega^s$ is the unique maximizer of $E$ relative to $\mathcal R_{\omega^s}$.
It follows that $\omega^s=v$ for any $v\in  \mathcal O_{\omega^s}$, which means that  $\omega^s$ is radial. The stability assertion is a straightforward consequence of Theorem \ref{thm1}.
\end{proof}

The rest of this section is devoted to the proof  of Theorem \ref{thm3}(ii). From now on, let $\omega^s$ be as given in Theorem \ref{thm3}(ii).

\begin{lemma}\label{lm661}
  There exists some radial function $\omega^r$ and some $\omega^e\in\mathbf E_1$ such that
$\omega^s=\omega^r+\omega^e$.
\end{lemma}
\begin{proof}
  Since $\omega^s(r,\theta)$ is a maximizer,  by rotational symmetry $\omega^s(r,\theta+\alpha)$ is also a maximizer for any $\alpha\in\mathbb R$. By Proposition \ref{propm2},
  \[\omega^s(r,\theta+\alpha)-\omega^s(r,\theta)\in\mathbf E_1\quad\forall\,\alpha\in\mathbb R.\]
 Combining this with the fact that $\mathbf E_1$ is a closed subspace of $L^2(D)$, we obtain $\partial_\theta \omega^s \in\mathbf E_1.$  In view of \eqref{dsk1}, there exist constants $a,b,c\in\mathbb R$   such that
  \[\partial_\theta\omega^s(r,\theta)=a J_0(j_{1,1}r)+bJ_1(j_{1,1}r)\cos\theta+cJ_1(j_{1,1}r)
  \sin\theta.\]
  Therefore,
  \begin{align*}
  \omega^s(r,\theta)&=\omega^s(r,0)+\int_0^\theta a J_0(j_{1,1}r)+bJ_1(j_{1,1}r)\cos\tau+cJ_1(j_{1,1}r)\sin\tau d\tau\\
  &=\omega^s(r,0)+a J_0(j_{1,1}r)\theta+ bJ_1(j_{1,1}r)\sin\theta+ cJ_1(j_{1,1}r)(1-\cos\theta).
  \end{align*}
To ensure that $\omega^s$ is $2\pi$-periodic in $\theta$, we must have $a=0$. The desired result then follows immediately.
\end{proof}

Denote
 \[\mathcal I_{\omega^s}:=\left\{v\in L^1(D) \;\middle|\; I(v)=I(\omega^s)\right\}.\]

\begin{lemma}\label{lm662}
The steady flow with vorticity $\omega^s$ is stable in the following sense:
for any $\varepsilon>0$, there exists some $\delta>0,$ such that for any smooth Euler flow with vorticity $\omega$, if
 \[\|\omega(0,\cdot)-\omega^s\|_{L^p(D)}<\delta,\]
 then for any $t\in\mathbb R,$ there exists some $\tilde\omega\in (\omega^s+\mathbf E_1)\cap \mathcal R_{\omega^s}\cap \mathcal I_{\omega^s}$
  such that
 \[ \|\omega(t,\cdot)-\tilde\omega\|_{L^p(D)}<\varepsilon.\]
 \end{lemma}
\begin{proof}
Fix a sequence of smooth solutions $\{\omega_n\}$ to the Euler equations and a sequence of moments $\{t_n\}\subset\mathbb R$ such that
$\omega_n(0,\cdot)\to  \omega^s$ in $L^p(D)$
as $n\to\infty.$
By  Theorem \ref{thm2}, we know that, up to a subsequence,
$\omega_n(t_n,\cdot)\to \tilde\omega$ in $L^p(D)$
as $n\to\infty$ for some $\tilde \omega\in  (\omega^s+\mathbf E_1)\cap \mathcal R_{\omega^s}$.   Since $I$ is conserved under the Euler dynamics, we have that
\[I(\omega_n(0,\cdot)) =I(\omega_n(t_n,\cdot))\quad\forall\,n.\]
Passing to the limit $n\to\infty$ yields $I(\tilde\omega)=I(\omega^s)$, which implies $\tilde\omega\in \mathcal I_{\omega^s}.$ This completes the proof.

\end{proof}

\begin{lemma}\label{lm663}
It holds that
 \[\mathcal O_{\omega^s}=(\omega^s+\mathbf E_1)\cap\mathcal R_{\omega^s}\cap \mathcal I_{\omega^s}.\]

\end{lemma}

\begin{proof}
The inclusion
\[
(\omega^s + \mathbf E_1) \cap \mathcal R_{\omega^s} \cap \mathcal I_{\omega^s} \supset \mathcal O_{\omega^s}
\]
follows from \eqref{rotainv} and Proposition~\ref{propm2}.
Below, we prove the reverse inclusion
\[
(\omega^s + \mathbf E_1) \cap \mathcal R_{\omega^s} \cap \mathcal I_{\omega^s} \subset \mathcal O_{\omega^s}.
\]
Fix  an arbitrary $v\in(\omega^s+\mathbf E_1)\cap\mathcal R_{\omega^s}\cap \mathcal I_{\omega^s}.$  Let $\omega^r$ be the radial function in Lemma \ref{lm661}. Then $\omega^s$ and $v$ can be expressed as
    \[
    \omega^s=\omega^r+aJ_0(j_{1,1}r)+bJ_1(j_{1,1}r)\cos(\theta+\beta),
  \,\,\,
       v=\omega^r+a'J_0(j_{1,1}r)+b'J_1(j_{1,1}r)\cos(\theta+\beta'),
       \]
   where $a,b,\beta,a',b',\beta'\in\mathbb R.$  Without loss of generality, we may assume that \( b, b' \geq 0 \) (this can be achieved by choosing new \(\beta\) and \(\beta'\)).
 It suffices to show
   $a=a'$ and $b=b'$.
   Since $v\in\mathcal I_{\omega^s}$, we have that
   \begin{equation}\label{cpq01}
   \int_D|\mathbf x|^2v(\mathbf x)d\mathbf x=\int_D|\mathbf x|^2\omega^s(\mathbf x)d\mathbf x.
   \end{equation}
  Inserting the expressions for $\omega^s$ and $v$ into \eqref{cpq01} and computing the integrals in polar coordinates,
   we obtain
   \[a\int_0^1J_0(j_{1,1}r)r^3dr=a'\int_0^1J_0(j_{1,1}r)r^3dr.\]
To proceed, we claim that
 \begin{equation}\label{lst}
\int_0^1 J_0(j_{1,1}r)r^3 dr\neq 0.
\end{equation}
In fact,
\[
 \int_0^1 J_0(j_{1,1}r)r^3 dr =\frac{1}{j_{1,1}^4}\int_0^{j_{1,1}}J_0(r)r^3 dr
  =-\frac{2}{j_{1,1}^4}\int_0^{j_{1,1}} J_1(r)r^2dr
 <0.
\]
where we used the facts that $(J_1(r)r)'=J_0(r)r$ (see \cite[p. 93]{Bo}) and that $J_1$ is positive in $(0,j_{1,1})$.
From \eqref{cpq01} and \eqref{lst}, we obtain
$
a=a'.
$
Therefore, $\omega^s$ and $v$ can be written as
 \begin{equation}\label{womnf}
\omega^s=\omega^r+aJ_0(j_{1,1}r)+bJ_1(j_{1,1}r)\cos(\theta+\beta),\,\,\,v=\omega^r+aJ_0(j_{1,1}r)+b'J_1(j_{1,1}r)\cos(\theta+\beta').
\end{equation}
Since $v\in\mathcal R_{\omega^s}$, we have
\begin{equation}\label{l2nm}
\int_Dv^2d\mathbf x=\int_D(\omega^s)^2d\mathbf x.
\end{equation}
Inserting \eqref{womnf} into \eqref{l2nm}, a straightforward computation gives
  $b^2=b'^2$,
and thus $b=b'$ (noting that we have assumed $b,b'\geq 0$). The proof is complete.

\end{proof}

\begin{proof}[Proof of Theorem \ref{thm3}(ii)]
It follows from Lemmas \ref{lm661}-\ref{lm663}.
\end{proof}

\bigskip
\noindent{\bf Acknowledgements:} G. Wang is supported by National Natural Science Foundation of China (12471101) and the Fundamental Research Funds for the Central Universities (DUT23RC(3)077).

\bigskip
\noindent{\bf  Data availability statement} All data generated or analysed during this study are included in this published article.

\bigskip
\noindent{\bf Conflict of interest}  The authors declare that they have no conflict of interest to this work.

\phantom{s}
 \thispagestyle{empty}

\end{document}